\documentclass[a4paper]{amsart}

\usepackage{amssymb}

\usepackage{graphicx}

\usepackage{amsfonts}
\usepackage{amssymb}
\usepackage{nicefrac}
\usepackage{enumerate}
\usepackage{url,bm}
\usepackage{color}      
\usepackage{capt-of}
\usepackage[svgnames]{xcolor}

\newcommand{\R}{\mathbb{R}}

\newcommand{\N}{\mathbb{N}}
\newcommand{\C}{\mathbb{C}}
\newcommand{\Z}{\mathbb{Z}}
\newcommand{\CC}{\mathcal{C}}
\newcommand{\F}{\mathcal{F}}
\newcommand{\A}{\mathcal{A}}

\newcommand{\T}{\mathcal{T}}

\newcommand{\calH}{\mathcal{H}}
\newcommand{\calP}{\mathcal{P}}
\newcommand{\calC}{\mathcal{C}}
\newcommand{\calD}{\mathcal{D}}
\newcommand{\calE}{\mathcal{E}}
\newcommand{\calK}{\mathcal{K}}
\newcommand{\calQ}{\mathcal{Q}}
\newcommand{\calV}{\mathcal{V}}

\newcommand{\calN}{\mathcal{N}}
\newcommand{\X}{\mathcal{X}}

\newcommand{\w}{\omega}

\newcommand{\la}{\lambda}
\newcommand{\g}{\gamma}
\newcommand{\W}{\Omega}

\newcommand{\wh}[1]{\widehat{#1}}
\newcommand{\wt}[1]{\widetilde{#1}}

\newtheorem{theorem}{Theorem}[section]
\newtheorem{lemma}[theorem]{Lemma}

\theoremstyle{definition}
\newtheorem{definition}[theorem]{Definition}

\newtheorem{corollary}[theorem]{Corollary}
\newtheorem{proposition}[theorem]{Proposition}

\theoremstyle{remark}
\newtheorem{remark}[theorem]{Remark}

\numberwithin{equation}{section}

\begin{document}

\title[An Approximation Problem in Multiplicatively Invariant Spaces]{An Approximation Problem in Multiplicatively Invariant Spaces}


\author{C. Cabrelli}
\address{Departamento de Matem\'atica,
Facultad de Ciencias Exac\-tas y Naturales,
Universidad de Buenos Aires, Ciudad Universitaria, Pabell\'on I,
1428 Buenos Aires, Argentina and
IMAS-CONICET, Consejo Nacional de Investigaciones
Cient\'ificas y T\'ecnicas, Argentina}
\email{cabrelli@dm.uba.ar}

\author{C. A. Mosquera}
\address{Departamento de Matem\'atica,
Facultad de Ciencias Exac\-tas y Naturales,
Universidad de Buenos Aires, Ciudad Universitaria, Pabell\'on I,
1428 Buenos Aires, Argentina and
IMAS-CONICET, Consejo Nacional de Investigaciones
Cient\'ificas y T\'ecnicas, Argentina}
\email{mosquera@dm.uba.ar}

\author{V. Paternostro}
\address{Departamento de Matem\'atica,
Facultad de Ciencias Exac\-tas y Naturales,
Universidad de Buenos Aires, Ciudad Universitaria, Pabell\'on I,
1428 Buenos Aires, Argentina and
IMAS-CONICET, Consejo Nacional de Investigaciones
Cient\'ificas y T\'ecnicas, Argentina}
\email{vpater@dm.uba.ar}
\thanks{The three authors were supported in part by Grants: CONICET PIP 11220110101018, PICT 2011-436, 
UBACyT 20020130100403BA, UBACyT 20020130100422BA and the third author was also supported by a Return Fellowship for Postdoctoral Researcher of the Alexander von Humboldt Foundation}

\subjclass[2000]{Primary 42C40, 94A12 ; Secondary 42C15, 47A15, 43A70}


\keywords{Shift-invariant spaces, Extra invariance, Multiplicatively invariant spaces, Approximation}

\begin{abstract}
Let $\calH$ be Hilbert space  and $(\W,m)$ a $\sigma$-finite measure space. Multiplicatively invariant (MI) spaces are closed subspaces of $ L^2(\Omega, \calH)$
that are invariant under point-wise multiplication by  functions from a fixed subset of $L^{\infty}(\Omega).$
Given a finite set of data $\mathcal{F}\subseteq L^2(\Omega, \calH),$ in this paper we prove the existence and construct  an MI space $M$ that best fits $\mathcal{F}$, in the least squares sense. 
MI spaces are related to  shift-invariant (SI) spaces via a fiberization map, which allows us to solve an approximation problem for SI spaces in the context of locally compact abelian groups.
On the other hand, we introduce the notion of decomposable MI spaces (MI spaces that can be decomposed into an orthogonal sum of MI subspaces)
and solve the approximation problem for the class of these spaces. Since  SI spaces
having extra invariance are in one-to-one relation to decomposable MI spaces, we  also solve our approximation problem for this class of SI spaces.
Finally we prove that  translation-invariant spaces  are in correspondence with
totally decomposable MI spaces.
\end{abstract}
\maketitle


\section{Introduction}
It is a common problem in applications to try to find a simple model to represent  a set of data $\mathcal{F} = \{f_1,...,f_m\}$ that is supposed to be included in some
Hilbert space $\mathcal{H}$. This becomes very important in particular when $m$ is very  large. 

Although the data are usually assumed to come from a low dimensional subspace, the low dimensionality is often lost because  the data  have been corrupted by noise. Therefore, one  wants to find a low dimensional subspace that best fits the data $\F.$

Usually some a priori assumptions are made about the data. For example that they are ``band-limited", that is,  they belong to a Paley-Wiener
space. Then one can use the Whittaker-Nyquist-Kotelnikov-Shannon Sampling Theorem. This is a common practice in engineering, in particular in applications to signal
processing.  However there are applications where these assumptions are not realistic.

There is another approach to the problem. Instead of fixing a model  based on an a priori hypothesis on the data, we can use the data to obtain
a  model (subspace) that is, in some sense,  the most suitable for them. This optimal model is  chosen  from a ``nice" class $\mathcal{C}$ of subspaces.

So a  mathematical description of this problem is: 
Given a Hilbert space $\calH$ and  a class $\mathcal{C}$ of closed subspaces of $\calH,$ try to prove the existence and find for every finite set $\F \subseteq \calH$ an element  $S^*\in \calC$
such that $\mathcal{E}(\F,S^*)\leq \mathcal{E}(\F,S)$ for all $S \in \mathcal{C}.$ Here $\mathcal{E}$ is some conveniently chosen functional.

This problem has been studied in different scenarios. The general problem of {\it existence} was studied in the abstract setting
by Aldroubi and Tessera in \cite{AT11}. They found necessary and sufficient conditions on a class $\CC$ of subspaces in some Hilbert space $\calH$
for the existence of an optimal solution
for the functional  $\mathcal{E}(\F,S) = \sum_{j=1}^m \|f_j-P_Sf_j\|^2.$ However they did not provide a way to construct it. See also \cite{CL15} for a version in Banach spaces.

It is important for applications to have a description of the solutions and to have an estimation of the error, that is the value of $\mathcal{E}(\F,S^*)$ for  $S^*$ being the  optimal subspace.

When $\calH$ is the Euclidean space $\R^d$ and  $\CC$ is the class of  subspaces of dimension less or equal than some 
positive integer $\ell,$ the problem can be tackled by an application of the Eckard-Young Theorem using the Singular Value Decomposition of a matrix, constructed from the data and was solved in \cite{EY36}. See also
\cite{Sch07}.

A case that is mathematically very challenging is when $\calH$ is the Lebesgue space $L^2(\R^d)$ and   $\calC$ is the class of  {\it shift-invariant} (SI) spaces with few generators.
Let $\Gamma$ be a full lattice in $\R^d$. A closed subspace $V\subseteq L^2(\R^d)$ is said to be {\it shift invariant} with respect to the lattice $\Gamma$ if $T_\gamma f \in V$ for every
$f \in V$ and every $\gamma \in \Gamma.$ Here $T_{\gamma}$ is the translation operator by $\gamma$.
The class of SI spaces has been studied in depth and they have proved to be very useful in many applications such as wavelets, sampling, approximation theory, harmonic analysis and signal processing. (See, e.g., \cite{AG01, Gro01, HW96, Mal89} and references therein).

The problem of approximation mentioned above was solved completely for the case of SI spaces in  $L^2(\R^d)$ in \cite{ACHM07}.
The authors constructed the generators of an optimal subspace and provided a formula that gives exactly the error of approximation in terms
of the spectrum of an associated matrix of periodic functions.
The key element in the argument is to transform the SI space into a {\it multiplicatively invariant} (MI) space using the Fourier transform composed with another isometry,
that  is called {\it the fiberization map}.
In a nutshell, if $\calH$ is a Hilbert space and $(\Omega, m)$ is a $\sigma$-finite measure space, a {\it multiplicatively invariant} space with respect to $D$, is a closed subspace of 
 $L^2(\Omega,\calH)$ that is invariant by point-wise multiplication by functions in $D$.  Here the set $D$ is a subset of  $L^{\infty}(\Omega,\mu)$ satisfying a certain property 
(see Section \ref{sec:MI-space} for the precise definition).

MI spaces have been studied and characterized by Helson \cite{Hel64, Hel86}, Srinivasan \cite{Sri64} and Hasumi-Srinivasan \cite{HS64} in the sixties as a generalization
    of Wiener's Theorem. (See Proposition \ref{prop:wiener}). The characterization is in terms of {\it range functions,} a concept introduced by Helson at that time.(See \cite[Chapter I, $\mathsection{3},$ Thm. 1]{Hel86}).
They were called {\it doubly invariant spaces}. 
Taking advantage of these results  Bownik and Ross \cite{BR14} gave a characterization for  subspaces invariant under translations along a subgroup which is not assumed to be discrete.

In \cite{dBDVRI94}, the authors used this characterization to establish the fiberization map that associates to every SI space  an MI space 
defined by the corresponding range function (see also \cite{RS95}). They studied the existence of bases and quasi-bases of translates in 
finitely generated SI spaces,  i.e. SI spaces  $V$ for which  there exists a finite set of functions $\A$ in $V$ with $V=\overline{\mbox{span}}\{T_{\gamma}\varphi:  \varphi \in \A,\;  \gamma \in \Gamma\}$.
Bownik in \cite{Bow00} used this isometric isomorphism  to characterize frames and Riesz bases of SI spaces.

This strategy of studying questions about SI spaces in the associated MI spaces and its range functions, is called {\it fiberization } and has been very fruitful,
in particular in the approximation problem mentioned above.

In this paper we  study this approximation problem, in  general MI spaces. 
  More precisely, given a Hilbert space $\calH$ and a finite set of functions $\mathcal{F}$ in  $L^2(\Omega,\calH)$, we want to 
 find an MI subspace of $L^2(\Omega,\calH)$ in a certain class that best fits the functions in $\mathcal{F}$.
 
The reason to choose this setting is because not only can one  derive   results for   SI spaces, but also MI spaces are much more general
and the problem can be important in  other situations.
For example, if instead  of SI spaces we consider spaces invariant under a group of unitary operators coming from the action of a discrete abelian group
in a general measure space, these spaces can be mapped isomorphically and isometrically to  MI spaces through the generalized Zak transform (see \cite{BHP15}). Also, MI spaces and their range functions are relevant in multiplicity theory  for normal operators (see \cite{Hel86}).

On the other side,  SI and MI spaces and the study of their structure using range functions extend naturally to locally compact abelian (LCA) groups, \cite{CP10}, 
(see also \cite{BR14} for more general lattices). Therefore with our approximation results on MI spaces  we shall, on one hand, obtain the existing approximation results on $L^2(\R^d)$ (\cite{ACHM07, CM15}) and on the other,  their generalization to the LCA group setting. 

An interesting and useful instance of the approximation problem is when one considers the elements of the approximating class to be SI spaces with {\it extra invariance}.
Let $V$ be an SI space in $L^2(\R^d)$ with respect to a full lattice $\Gamma$ and assume that $\Lambda$ is another full lattice
containing $\Gamma.$ The SI space $V$ is then said to have {\it $\Lambda$-extra invariance, }  if $T_\lambda f\in V$ for every $f\in V$ and every $\lambda\in\Lambda$.
These  spaces with extra invariance have been completely characterized, first in the real line  \cite{ACHKM10}, then in $\R^d$ \cite{ACP10} and more generally in LCA groups \cite{ACP11}.  See also \cite{SW11} for an application to sampling.

Recently, the approximation problem for the class of SI spaces with extra invariance was solved in  \cite{AKTW12},
for  principal shift-invariant spaces in one variable and under the assumption that the space has  one generator with orthogonal integer translates.
The general case for frame generators in several variables was solved in \cite{CM15}.

The extra invariance of an  SI space is reflected into the associated MI space by the fact that 
the MI space can be decomposed into an  orthogonal  sum of MI subspaces. This decomposition comes from  a decomposition of the Hilbert space $\calH$ in which the functions of the MI space take values.
This property suggest  to define the class of {\it decomposable MI spaces}.
So, to each SI space with extra invariance there corresponds a particular decomposable MI space.
However, decomposable MI spaces is a richer class and for general $\Omega$ and $\calH$  may not be associated with an SI space.

 Once we introduce this new concept of decomposable MI spaces, we study some of their properties and give a characterization of them in terms of their associated range functions.
Consequently, in this paper we also consider the approximation problem for decomposable  MI subspaces of $L^2(\Omega,\calH)$. Furthermore, we establish a connection between {\it totally} decomposable MI spaces and SI spaces having total extra invariance, that is, SI spaces invariant under every translation.

The paper is organized as follows. In Section \ref{preliminaries} we provide the general setting of MI spaces and the definitions we shall use in the following sections. In Section \ref{sec:extra-MI-spaces} we introduce the concept of decomposable MI spaces, study some of their properties and give a characterization in terms of their associated range function. The approximation problems are treated in Section \ref{sec:approximation-problems}. First we prove the existence of an optimal MI space in Subsection \ref{sec:optimal-MI}. In Subsection \ref{sec:approx-on-H} we study an abstract version of the problem in general Hilbert spaces and in Subsection \ref{sec:optimal-decomposable-MI} we consider the problem  for decomposable MI spaces. 
The connection to the setting of shift-invariant spaces on LCA groups is provided in Section \ref{sec:shift-invariant}.
Finally, in Section \ref{sec:PW-MI-spaces} we analyze the connection between those MI spaces that are totally decomposable with SI spaces having total extra invariance.

\bigskip

\section{Preliminaries}\label{preliminaries}

We begin with a review of some basic results and definitions that we will use in subsequent sections. 
The known results are generally stated without proofs, but we provide references where the proofs can be found. 

\subsection{Multiplicatively invariant spaces in $L^2(\Omega, \mathcal{H})$} \label{sec:MI-space}
\ 

\medskip
This section summarizes all  we need about multiplicatively invariant spaces and it is  extracted from \cite[Section 2]{BR14}, where the reader can find the proofs of the results. See also \cite{Hel64, Sri64, HS64}. 

Let $(\Omega, \mu)$ be a $\sigma-$finite measure space and $\mathcal{H}$ be a separable Hilbert space.  
The vector-valued space $L^2(\Omega, \mathcal{H})$ is defined as
\[
L^2(\Omega, \mathcal{H}):=\{\Phi\colon\Omega\to\mathcal{H}\colon \Phi \mbox{ is measurable, } \|\Phi\|^2= \int_{\Omega} \|\Phi(\w)\|_{\mathcal{H}}^2 \, d\mu(\w) <\infty\}.
\]
It is a Hilbert space with  inner product  given by 
\[
\langle\Phi, \Psi\rangle=\int_{\Omega}\langle\Phi(\w), \Psi(\w)\rangle_{\mathcal{H}}  \, d\mu(\w).
\]
In order to state  the definition of multiplicatively invariant space (MI spaces) in $L^2(\Omega, \mathcal{H}),$  we need to recall  the notion of determining set for 
$L^1(\Omega)$, introduced in \cite{BR14}.

\begin{definition}\label{def:determining-set}
A set $D\subseteq L^{\infty}(\Omega)$ is a {\it determining set for $L^1(\Omega)$} if for every $f\in L^1(\Omega)$ such that $\int_{\Omega} fg\, d\mu=0$ for all $g\in D,$ one has $f=0.$
\end{definition} 

\begin{definition}\label{def:MI}
A closed subspace $M\subseteq L^2(\Omega, \mathcal{H})$ is {\it multiplicatively invariant} with respect to the determining set $D$ 
for $L^1(\Omega)$ (briefly, $D$-MI) if 
\[
\Phi\in M \Longrightarrow g\Phi\in M, \quad \mbox{ for any }g\in D.
\] 
\end{definition}

For an at most  countable subset $\A\subseteq L^2(\Omega, \mathcal{H}),$ we define the multiplicatively invariant space with respect to $D$ generated
by $\A$ as
\[
M_D(\A):=\overline{\mbox{span}}\{g\Phi\colon \Phi\in\A, g\in D\}.
\]
When $\A$ is finite, $M_D(\A)$ is said to be finitely generated by $\A.$ In general, we shall omit to make reference to the underlying determining set and  simply say ``MI spaces''.    For a finitely generated MI space $M$, we define its length as 
\[
\ell(M):=\min \{n\in\N\colon \exists \,\Phi_1,\dots, \Phi_n\in M, \mbox{ with } M=M_D(\Phi_1,\dots,\Phi_n)\}.
\] 
 
One of the most important properties of MI spaces is their characterization through  range functions, which we state in the next theorem (see \cite[Theorem 2.4]{BR14} and \cite[Chapter I, $\mathsection{3}$, Thm. 1]{Hel86}.  First, recall that a {\it range function} is a mapping $J\colon \Omega\to \{\mbox{closed subspaces of }\mathcal{H}\}.$ For $\w\in\Omega,$ the orthogonal projection of $\mathcal{H}$ onto $J(\w)$ is denoted by 
$P_{J(\w)}.$  A range function is {\it measurable} if for every $a, b\in\mathcal{H},$ $\w\mapsto \langle P_{J(\w)}(a),b\rangle$ is measurable
as a function from $\Omega$ to $\C.$

\begin{theorem}\label{thm:MI-range-function}
Suppose that $L^2(\Omega)$ is separable, so that $L^2(\Omega, \mathcal{H})$ is also separable. Let $M$ be a closed subspace of $L^2(\Omega, \mathcal{H})$
and $D$ a determining set for $L^1(\Omega).$ Then, $M$ is a $D$-MI space if and only if there exists a measurable range function $J$ such that 
\begin{equation}\label{eq:M-J}
M= \{\Phi\in L^2(\Omega, \mathcal{H})\colon \Phi(\w)\in J(\w), \mbox{ a.e. } \w\in\Omega\}.
\end{equation}
Identifying range functions that are equal almost everywhere, the correspondence between $D$-MI spaces and measurable range functions is one-to-one and onto.

Moreover, when $M=M_D(\A)$ for some at most countable set $\A\subseteq L^2(\Omega, \mathcal{H})$ the range function associated to $M$ is
\[
J(\w)=\overline{\mbox{span} }\{\Phi(\w)\colon \Phi\in\A\}, \quad \mbox{ a.e. }\w\in\Omega.
\]  
\end{theorem}

\begin{remark}
From Theorem \ref{thm:MI-range-function} it follows that $M\subseteq L^2(\W,\calH)$ is a $D$-MI space if and only if $M$ is an $L^\infty(\W)$-MI space.  Indeed, if $M$ is a $D$-MI space, by Theorem \ref{thm:MI-range-function} $M$ has the form given by \eqref{eq:M-J} and then it is easily seen that $M$ is invariant under point-wise multiplication of  any function in $L^\infty(\W)$. The converse is obvious.  
\end{remark}

The next lemma will be important in what follows. It is a version of a result of Helson \cite{Hel64}  and was proven before within other contexts \cite{Bow00, CP10}. For a proof of it we refer to \cite[Proposition 2.2]{BR14}.
\begin{lemma}\label{lem:helson}
Let $M\subseteq L^2(\W,\calH)$ be a $D$-MI space with associated measurable range function $J$. Then, for every $\Phi\in L^2(\W,\calH)$
$$(P_M\Phi)(\w)=P_{J(\w)}(\Phi(\w)),$$
for a.e. $\w\in \W$. Here $P_M$ is the orthogonal projection of $L^2(\W, \calH)$ onto $M$ and $P_{J(\w)}$ are the orthogonal projections associated to $J$. 
\end{lemma}

\subsection{Frames and uniform frames}
We recall that for a separable Hilbert space $\calH$ and an at most countable index set $I$, we say that $\{f_i\}_{i\in I}\subseteq \calH$ is a {\it frame for $\calH$} if there there exist constants $0<A\leq B$ such that 
\begin{equation}\label{eq:def-frame}
 A\|f\|^2_\calH\leq\sum_{i\in I}|\langle f, f_i\rangle_\calH|^2\leq B\|f\|^2_\calH,
\end{equation}
for all $f\in \calH$. The constants $A$ and $B$ are called {\it frame bounds}. When \eqref{eq:def-frame} holds for all $f\in \overline{\mbox{span} }\{f_i:\,i\in I\}$, we say that $\{f_i\}_{i\in I}$ is a {\it frame sequence}. When $A=B=1$, $\{f_i\}_{i\in I}$ is said to be a Parseval frame (or frame sequence). For details on the theory of frames in Hilbert spaces we refer the reader to \cite{Cas00, Chr03, Hei11} and the references therein.

When working with MI spaces, we will consider uniform frames, which are defined as follows:
\begin{definition}{\cite[Definition 3.5]{Pat15}}
 Let $\A\subseteq L^2(\Omega, \mathcal{H})$ be an at most countable set  and let $J$ be the  measurable range function  defined as $J(\w)=\overline{\mbox{span} }\{\Phi(\w):\Phi\in\A\}, \quad \textrm{a.e. }\w\in\W$. We say that $\A$ is a {\it uniform frame} for $J$ if there exist constants $0<A\leq B$ such that, for a.e. $\w\in\W$,  the set $\{\Phi(\w):\,\Phi\in\A\}$ is a frame for $J(\w)$ with frame bounds $A$ and $B$.
\end{definition}

\section{Decomposable MI spaces in $L^2(\Omega, \mathcal{H})$} \label{sec:extra-MI-spaces}
In this section we will introduce the notion of decomposable MI spaces. For this, fix  $D$  a determining set for $L^1(\Omega)$ and for a given $\kappa\in\N\cup\{+\infty\}$,  suppose that $\calH$ can be decomposed into an  orthogonal sum as 
\begin{equation}\label{eq:descomposicion}
\mathcal{H}=\mathcal{H}_1\oplus\cdots\oplus \mathcal{H}_{\kappa}.
\end{equation}
For each $D$-MI subspace $M,$ we define 
\[
M_j:= \{\calP_j(\Phi)\colon \Phi\in M\}, \mbox{ for all } 1\le j\le \kappa,
\]
where $\calP_j(\Phi)$ denotes the function in $L^2(\Omega, \mathcal{H})$ defined as $\calP_j(\Phi)(\w):=P_{\calH_j}(\Phi(\w))$  with  $P_{\mathcal{H}_j}\colon \mathcal{H}\to\mathcal{H}_j$ being  the orthogonal projection of $\calH$ onto $\calH_j$. Note that $\calP_j(\Phi)$ is indeed an element in
$L^2(\Omega, \mathcal{H})$ since it is measurable because  $\w\mapsto\langle \calP_j(\Phi)(\w), a\rangle=\langle \Phi(\w), P_{\mathcal{H}_j} a\rangle$ is measurable in the usual sense for every $a\in\calH$ and $\|\calP_j(\Phi)\|\leq \|\Phi\|$.
Moreover, it is easily seen that $\calP_j$ is the orthogonal projection of $L^2(\Omega, \mathcal{H})$ onto the subspace $\{\Phi\in L^2(\Omega, \mathcal{H})\,:\, \Phi(\w)\in \calH_j, \textrm{ a.e. }\w\in\W\}$, which we identify with $L^2(\Omega, \mathcal{H}_j)$. Note that, with this identification, we have $L^2(\Omega, \mathcal{H})=L^2(\Omega, \mathcal{H}_1)\oplus\cdots\oplus L^2(\Omega, \mathcal{H}_\kappa)$.

\begin{definition}\label{def:decomposable-MI}
 We say that a $D$-MI subspace $M\subseteq L^2(\Omega, \mathcal{H})$ is {\it decomposable} with respect to  $\{\calH_1, \ldots, \calH_\kappa\}$ if 
\[
M_j\subseteq M \quad \mbox{ for all } 1\le j\le\kappa. 
\]
For the particular case when the decomposition of $\calH$ is given by an orthonormal basis $\{\delta_j\}_{j\in\N}$, i.e. $\calH_j=\mbox{span}\{\delta_j\}$ we say that $M$ is {\it totally} decomposable with respect to $\{\mbox{span}\{\delta_j\}\}_{j\in\N}$. 
\end{definition}

Note that, when $M$ is decomposable with respect to $\{\calH_1, \ldots, \calH_\kappa\}$, we have that  $M$ is the orthogonal sum $M=M_1\oplus\cdots\oplus M_{\kappa}$.

\begin{remark}\noindent
\begin{enumerate}
 \item 
Given any closed subspace $\calK\subseteq \calH$ there exists a MI space $M$ which is decomposable with respect to  any  decomposition of $\calH$ as \eqref{eq:descomposicion} having $\calK$ as one of its components, that is, with $\calK=\calH_j$ for some $j$. Indeed, let $A$ be a measurable subset of $\W$ with $0<\mu(A)<+\infty$ and $x\in \calK$, $x\neq0$. Define $\Phi=x\chi_A$ where $\chi_A$ denotes the characteristic function of $A$ and consider $M=M_D(\Phi)$. Then, since $\calP_j(\Phi)=\Phi$ and $\calP_r(\Phi)=0$ for $r\neq j$, it follows that $M$ is decomposable with respect to $\{\calH_1, \ldots, \calH_\kappa\}$, where $\mathcal{H}=\mathcal{H}_1\oplus\cdots\oplus \mathcal{H}_{\kappa}$ and $\calK=\calH_j$ for some $j$.
 \item This definition of decomposable MI spaces is inspired in the definition of extra invariance of shift-invariant spaces, \cite{ACHKM10, ACP10, ACP11}. As we shall see in detail in Section \ref{sec:shift-invariant}, shift-invariant spaces are in one-to-one correspondence with MI spaces for a proper choice of $\W$ and $\calH$. Under this correspondence, shift-invariant spaces having extra invariance correspond to decomposable MI spaces with respect to a particular decomposition of $\calH.$    
\end{enumerate}
\end{remark}

In the next proposition we summarize other important properties of decomposable  MI spaces. In particular, we shall see that given a decomposable MI space $M$ the spaces $M_j$ of the previous definition are also multiplicatively invariant and show how their associated range function can be related to that of  $M$.

\begin{proposition}\label{prop:extra-invariance}
 Let $M\subseteq L^2(\Omega, \mathcal{H})$ be a $D$-MI space with associated range function $J$. Suppose that $\calH$ has a decomposition as in \eqref{eq:descomposicion} and that $M$ is decomposable with respect to  $\{\calH_1, \ldots, \calH_\kappa\}$. Then, for every $1\leq j\leq\kappa$ the following conditions hold:
 \begin{enumerate}
  \item [(i)] $M_j$ is a $D$-MI space.
  \item [(ii)] If $M=M_D(\A)$ for some at most countable set $\A\subseteq L^2(\Omega, \mathcal{H})$, then $M_j=M_D(\calP_j(\A))$.
  \item [(iii)] The measurable range function $J_j$ defined by $J_j(\w):=P_{\calH_j}(J(\w))$ is the one associated to $M_j$ through Theorem \ref{thm:MI-range-function}.
 \end{enumerate}
\end{proposition}

\begin{proof}
 $(i)$. Let $\Phi\in M$ and $g\in D$. Since, $g\calP_j(\Phi)=\calP_j(g\Phi)$ and $M$ is a $D$-MI space, we  conclude that $g\Phi\in M_j$ for every $g\in D$ and $\Phi\in M_j$. Therefore, to claim that $M_j$  is a $D$-MI space, we must show that it is closed. For this, take $\{\Psi_n\}_{n\in\N}\subseteq M_j$ converging to some $\Psi\in L^2(\Omega, \mathcal{H})$. Note that since $M_j\subseteq M$ and $M$ is closed, we have that $\Psi\in M$. Thus, 
 \begin{align*}
  \|\Psi_n-\Psi\|^2&=\int_\Omega\|P_{\calH_j}(\Psi_n(\w)-\Psi(\w))\|^2_\calH\,d\mu(\w)\,\, +
  \,\,\int_\Omega\|P_{\calH_j^{\perp}}(\Psi_n(\w)-\Psi(\w))\|^2_\calH\,d\mu(\w) \\
  &= \int_\Omega\|\Psi_n(\w)- P_{\calH_j}(\Psi(\w))\|^2_\calH\,d\mu(\w) +
  \int_\Omega\|P_{\calH_j^{\perp}}(\Psi(\w))\|^2_\calH\,d\mu(\w)\\
  &=\|\Psi_n-\calP_j(\Psi)\|^2\,+\,\|\calP_j^{\perp}(\Psi)\|^2,
 \end{align*}
where $\calP_j^{\perp}(\Psi)$ denotes the function $\calP_j^{\perp}(\Psi)(\w)=P_{\calH_j^{\perp}}(\Psi(\w))$ for a.e. $\w\in\Omega$. Now, since $\Psi_n\to\Psi$ as $n\to\infty$, we have that 
 $\Psi_n\to\calP_j(\Psi)$ as $n\to\infty$ and $\calP_j^{\perp}(\Psi)=0$. Hence, $\calP_j(\Psi)=\Psi$ and therefore, $\Psi\in M_j$ proving that $M_j$ is closed.
 
 $(ii)$. Clearly, $M_D(\calP_j(\A))\subseteq M_j$. To prove the other inclusion, suppose there exists $\Psi\in M_j$ orthogonal to $M_D(\calP_j(\A))$. Then, in particular we have that for every $\Phi\in\A$
 $$
  \langle\Psi, \calP_j(\Phi)\rangle=\int_\Omega\langle\Psi(\w), P_{\calH_j}(\Phi(\w))\rangle\,d\mu(\w)=\int_\Omega\langle\Psi(\w),\Phi(\w)\rangle\,d\mu(\w)= \langle\Psi, \Phi\rangle,
$$
which implies  that $\Psi\perp M$. Then $\Psi=0$ and the conclusion follows.

$(iii)$.  Let us call $K$ the measurable range function associated to $M_j$ through Theorem \ref{thm:MI-range-function}. Then, the result will be proven if we show that $K(\w)=J_j(\w)$ for a.e. $\w\in\Omega$. 

Using Theorem \ref{thm:MI-range-function} and item $(ii)$ of this proposition, we have that $K(\w)= \overline{\mbox{span} }\{\calP_j(\Phi)(\w)\colon \Phi\in\A\}$ for a.e. $\w\in\Omega$. Then, $K(\w)\subseteq J_j(\w)$ for a.e. $\w\in\Omega$. On the other hand, by the continuity of the orthogonal projections $P_{\calH_j}$, 
$J_j(\w)=P_{\calH_j}(\overline{\mbox{span} }\{\Phi(\w)\colon \Phi\in\A\})\subseteq 
\overline{\mbox{span} }\{P_{\calH_j}(\Phi(\w))\colon \Phi\in\A\}=K(\w)$ for a.e. $\w\in\Omega$. 
\end{proof}

We now provide a characterization of decomposable MI spaces  in terms of range functions. This result will be useful in the solution of the approximation 
problems. 

\begin{proposition}\label{lem:extra-invariance-range-function}
Let $M\subseteq L^2(\Omega, \mathcal{H})$ be a $D$-MI space with associated range function $J$. For $\kappa\in\N$ suppose that $\calH$ is decomposed as in \eqref{eq:descomposicion}. Then, the following two conditions are equivalent:
\begin{enumerate}
 \item [(i)] $M$ is decomposable with respect to  $\{\calH_1, \ldots, \calH_\kappa\}$. 
 \item [(ii)] For every $1\leq j\leq\kappa$, $J_j(\w)\subseteq J(\w)$ for a.e. $\w\in\Omega$, where $J_j(\w)=P_{\calH_j}(J(\w))$.
\end{enumerate}
\end{proposition}

\begin{proof}
 $(i)\Rightarrow(ii)$. Since by Proposition \ref{prop:extra-invariance} $J_j$ given by $J_j(\w)=P_{\calH_j}(J(\w))$ for a.e. $\w\in\Omega$ is the measurable range function associated to $M_j$ for every $1\leq j\leq\kappa$,  the result immediately follows.
 
 $(ii)\Rightarrow(i)$. Let  $1\leq j\leq\kappa$ and $\Psi\in M_j$. Then, there exists $\Phi\in M$ such that $\Psi=\calP_j(\Phi)$. Now, for a.e. $\w\in\Omega$ we have that $\Psi(\w)=P_{\calH_j}(\Phi(\w))\in P_{\calH_j}(J(\w))\subseteq J(\w)$ and this together with Theorem \ref{thm:MI-range-function} show that $\Psi\in M$.
\end{proof}

\section{Approximation problems}\label{sec:approximation-problems}
Throughout this section $(\W,m)$ will be a finite mesure space and the decomposition of $\calH$ will be as in \eqref{eq:descomposicion} for some $\kappa\in\N$. 
In this situation, we consider two approximation problems.  Both are in the same spirit: given a set $\F=\{F_1, \dots, F_m\}\subseteq L^2(\Omega, \mathcal{H})$ and a class $\calC$ of $D$-MI spaces, we look for an $M^*\in\calC$ which best fits the data $\F$ in the sense of least squares. This is, $M^*$ must satisfy that  
\begin{equation}\label{Problema}
\sum_{j=1}^m \|F_j- P_{M^*} F_j\|^2 \le \sum_{j=1}^m \|F_j- P_{M} F_j\|^2,
\end{equation}
for all $M\in\calC$. 

In the first version of this problem, we will minimize over the class $\calC_\ell$ which, for a given $\ell\in \N$ is defined as 
\begin{equation}\label{eq:class}
\mathcal{C}_{ \ell}:=\{M\subseteq L^2(\Omega, \mathcal{H})\colon M \mbox{ is }\text{D-MI}, \ell(M)\le\ell\}.
\end{equation} 
We shall refer to this particular version of the problem as {\it Problem 1}. 

In the second version of our approximation problem,  which will be referred  as {\it Problem 2}, we shall minimize over a subclass of $\calC_\ell$ defined as follows: 
given a decomposition of $\calH$ as in \eqref{eq:descomposicion}, $\mathcal{C}_{\{\calH_1, \ldots, \calH_\kappa\}, \ell}$ is the subclass of $\calC_\ell$ defined by 
\begin{equation}\label{eq:class-with-extra}
\mathcal{C}_{\{\calH_1, \ldots, \calH_\kappa\}, \ell}:=\{M\subseteq L^2(\Omega, \mathcal{H})\colon M \mbox{ is }\text{D-MI}, \ell(M)\le\ell, M_j\subseteq M \quad \forall \, 1\le j\le\kappa\}.
\end{equation} 
Note that $\mathcal{C}_{\{\calH_1, \ldots, \calH_\kappa\}, \ell}$ is the subclass of $\mathcal{C}_{\ell}$ of decomposable  MI spaces with respect to  $\{\calH_1, \ldots, \calH_\kappa\}$. 

The $L^2(\R^d)$-versions of Problems 1 and 2 were considered in \cite{ACHM07} and \cite{CM15} respectively, where the minimizing subspace was required to belong to certain class of shift-invariant spaces. The authors of \cite{ACHM07} showed a constructive way to find the minimizing shift-invariant space by means of the well-known Singular Value Decomposition.  
Recently, it was shown in \cite{CM15}, that in the context of shift-invariant spaces, Problem 2 has a solution.
We will see that it is not difficult to adapt the proofs of \cite{ACHM07, CM15}  to the setting of vector-valued functions and MI spaces. However, since this new setting is not as familiar as the shift-invariant space one and some details deserve to be carefully checked, we shall show here that Problems 1 and 2 have solutions, providing the complete proofs.     

For convenience of the reader we will state a result from \cite{ACHM07} that we will use in what follows.

\begin{theorem}\cite[Theorem 4.1]{ACHM07}\label{thm:optimal-Hilbert}
Let $\calH$ be an infinite dimensional Hilbert space, $\F=\{f_1, \dots, f_m\}\subseteq \calH$, $\mathcal{X}=\textrm{span}\{f_1, \dots, f_m\}$, $\la_1\geq\la_2\geq\cdots\geq\la_m$ the eigenvalues of the matrix $(\langle f_i, f_j\rangle)_{i,j}\C^{m\times m}$ and $y_1,\dots, y_m\in\C^m$, with $y_i=(y_{i1},\dots, y_{1m})^t$ orthonormal  left eigenvectors associated to the eigenvalues
$\la_1, \dots, \la_m$. Let $\ell\in\N$.

Define the vectors $q_1,\dots, q_\ell\in\calH$ by
$q_i=\la_i^{1/2}\sum_{j=1}^m y_{ij} f_j$
if $\la_i\neq0$ and $q_i=0$ otherwise. Then, $\{q_1,\dots, q_\ell\}$ is a Parseval frame for $W=\textrm{span}\{q_1,\dots, q_\ell\}$ and $W$ is optimal in the sense that 
$$\sum_{j=1}^m \|f_j- P_{W} f_j\|^2 \le \sum_{j=1}^m \|f_j- P_{S} f_j\|^2,$$
for all subspace $S$ such that $\dim(S)\leq \ell$.
Furthermore we have the following formula for the error
$$\mathcal{E}(\F, \ell)=\sum_{j=\ell+1}^m \la_j.$$
\end{theorem}

\subsection{Optimal MI spaces}\label{sec:optimal-MI}
\

\medskip
In this subsection, we shall give a solution to Problem 1.

\begin{theorem}\label{thm-problema1}
Let $\ell\in \N$ and $\F=\{F_1, \dots, F_m\}\subseteq L^2(\Omega, \mathcal{H})$. Then, there exists a $D$-MI space $M^*$ of length $\ell$ solving Problem 1, that is, satisfying
\begin{equation*}
\sum_{j=1}^m \|F_j- P_{M^*} F_j\|^2 \le \sum_{j=1}^m \|F_j- P_{M} F_j\|^2,
\end{equation*}
for all $M\in\calC_\ell$. 
\end{theorem}

For the proof of the above theorem we shall use the Gramian associated to a given set of functions of $L^2(\W, \calH)$ which we define as follows:
let $\F=\{F_1, \dots, F_m\}\subseteq L^2(\Omega, \mathcal{H})$, the {\it Gramian} associated to $\F$ is  the function $G_\F:\W\to\R^m\times\R^m$ defined a.e. $\w\in\W$ by 
$$(G_\F(\w))_{ij}=\langle F_i(\w), F_j(\w)\rangle_\calH.$$
For a.e. $\w\in\W$, $G_\F(\w)$ is a positive-semidefinite self-adjoint matrix and we denote and order its eigenvalues as $\lambda_1(\w)\geq \cdots\geq \lambda_m(\w)$. Since  the entries of $G_\F$ are measurable functions, by \cite[Lemma 4.3]{ACHM07} or \cite[Lemma 2.3.5]{RS95} the eigenvalues $\lambda_j:\W\to\R$, $1\leq j\leq m$, are measurable functions and there exists $U:\W\to\C^m\times\C^m$ with measurable entries and satisfying that $U(\w)$ is unitary for a.e. $\w\in\W$ such that  
$$G_\F(\w)=U(\w)\Lambda(\w)U^*(\w), \quad \textrm{ a.e. }\w\in\W,$$
where $\Lambda(\w):=\textrm{diag}(\lambda_1(\w), \ldots, \lambda_m(\w))$.
Note that, denoting by $U_j(\w)$ the $j$-th column of $U(\w)$, one can see that $y_j(\w):=U^*_j(\w)$ is the left-eigenvector of $G_\F(\w)$ associated to $\lambda_j(\w)$. Moreover, $\{y_j(\w)=(y_{j1}(\w),\ldots,y_{jm}(\w))\}_{j=1}^m$ is an orthonormal basis of $\C^m$, for a.e $\w\in\W$. 

\begin{proof}[Proof of Theorem \ref{thm-problema1}]
Keeping the notation described above, we will construct the generators of the minimizing $D$-MI space -- solution of Problem 1 --  using the eigenvalues and left-eigenvectors of $G_\F(\w)$. 

For $1\leq j\leq \ell$ and a.e. $w\in\W$, we define
\begin{equation}\label{eq:generators} 
 \Phi_j(\w):=\beta_j(\w)\sum_{i=1}^m \overline{y_{ji}}(\w)F_i(\w),
\end{equation}
where $\beta_j(\w):=\lambda_j(\w)^{-{1/2}}$ if $\lambda_j(\w)\neq0$ and $\beta_j(\w):=0$ if 
$\lambda_j(\w)=0$.

First, observe that $\Phi_j:\W\to\calH$ are measurable functions for $1\leq j\leq \ell$. Let us see that $\Phi_j\in L^2(\W, \calH)$. For this, using that $y_j(\w)$ are left-eigenvectors associated to $\lambda_j(\w)$ and that $\{ y_1(\w), \ldots, y_m(\w)\}$ forms an orthonormal basis of $\C^m$, for a.e. $\w\in \W$, we compute
\begin{align}\label{eq:L2-generators}
 \int_\W\|\Phi_j(\w)\|_\calH^2\,d\mu(\w)&=\int_\W\beta_j(\w)^2\|\sum_{i=1}^m \overline{y_{ji}}(\w)F_i(\w)\|_\calH^2\,d\mu(\w)\nonumber\\
 &=\int_\W\beta_j(\w)^2\sum_{s,i=1}^m \overline{y_{ji}}(\w)\langle F_i(\w), F_s(\w)\rangle_\calH y_{js}(\w)\,d\mu(\w)\nonumber\\
 &=\int_\W\beta_j(\w)^2\sum_{i=1}^m \overline{y_{ji}}(\w)\sum_{s=1}^m(G_\F(w))_{is}  y_{js}(\w)\,d\mu(\w)\nonumber\\
 &=\int_\W\beta_j(\w)^2\sum_{i=1}^m \overline{y_{ji}}(\w)\lambda_j(\w) y_{ji}(\w)\,d\mu(\w)\nonumber\\
 &=\int_\W\beta_j(\w)^2\lambda_j(\w)\,d\mu(\w)\nonumber\\
 &=\int_\W\chi_{\{\lambda_j>0\}}(\w)\,d\mu(\w)\leq \mu(\W)<+\infty. 
\end{align}

Now, consider $M^*:=M_D(\Phi_1,\ldots, \Phi_\ell)\in \calC_\ell$ and let us prove that $M^*$ is a solution for Problem 1. 

Let $J^*$ be the range function associated to $M^*$ through Theorem \ref{thm:MI-range-function}. For a.e. $\w\in\W$, since $J^*(\w)=\overline{\mbox{span} }\{\Phi_j(\w)\colon 1\leq j\leq \ell\}$,  by Theorem \ref{thm:optimal-Hilbert}, we know that $J^*(\w)$ is a subspace of $\calH$ that satisfies
$$\sum_{j=1}^m\|F_j(\w)-P_{J^*(\w)}(F_j(\w))\|^2_\calH\leq 
\sum_{j=1}^m\|F_j(\w)-P_{W}(F_j(\w))\|^2_\calH,$$
for all subspaces $W\subseteq \calH,$ with $\dim(W)\leq\ell.$ 
Then, using this and Lemma \ref{lem:helson}, we have that, for $M$ being a $D$-MI space of length $\ell$ and associated range function $J$,
\begin{align*}
 \sum_{j=1}^m\|F_j-P_{M^*}F_j\|^2
 &=\int_\W\sum_{j=1}^m\|F_j(\w)-(P_{M^*}F_j)(\w)\|^2_\calH\,d\mu(\w)\\
 &=\int_\W\sum_{j=1}^m\|F_j(\w)-P_{J^*(\w)}(F_j(\w))\|^2_\calH\,d\mu(\w)\\
 &\leq \int_\W\sum_{j=1}^m\|F_j(\w)-P_{J(\w)}(F_j(\w))\|^2_\calH\,d\mu(\w)\\
 &=\int_\W\sum_{j=1}^m\|F_j(\w)-(P_{M}F_j)(\w)\|^2_\calH\,d\mu(\w)\\
 &=\sum_{j=1}^m\|F_j-P_{M}F_j\|^2,
\end{align*}
and this proves what we wanted. 
\end{proof}

\begin{remark}\noindent
\begin{enumerate}
 \item Due to Theorem \ref{thm:optimal-Hilbert}, $\{\Phi_j(\w)\colon 1\leq j\leq \ell\}$ forms a Parseval frame for $J^*(\w)$ for a.e. $\w\in\W$. This means that the generators defined as in \eqref{eq:generators} of the optimal MI space 
$M^*$ form a uniform frame for $J^*$ with frame bounds $A=B=1$. A complete characterization of uniform frames for range functions can be found in \cite[Theorem 2.1]{Ive15}. 
\item The condition on $\W$ of having finite measure is crucial for ensuring that the generators defined in \eqref{eq:generators} are in $L^2(\W, \calH)$. However, one could delete that condition on $\W$ and instead ask  the data set $\F$ to be  such that $\mu(\{\w\in\W:\, G_\F(w)\neq 0\})<\infty$ (see \eqref{eq:L2-generators}). This also implies that the generators belong to   $L^2(\W, \calH)$ since $\{\w\in\W:\, G_\F(w)\neq 0\}=\{\w\in\W:\, \lambda_1(w)> 0\}$ and the condition $\mu(\{\w\in\W:\, \lambda_1(w)> 0\})<\infty$ gives $\mu(\{\w\in\W:\, \lambda_j(w)> 0\})<\infty$ for all $1\leq j\leq m$.
\item Since Theorem \ref{thm:optimal-Hilbert} provides a formula for the error, it follows that 
the error for Problem 1, which is defined as $ \calE(\F, \ell):=\min_{M\in \calC_\ell}\sum_{j=1}^m\|F_j-P_{M}F_j\|^2$, coincides with the expression $ \sum_{j=\ell+1}^m\int_\W\lambda_j(w)\,d\mu(w).$
\end{enumerate}
\end{remark}

\medskip

Now we will consider an approximation  problem for the Hilbert space $\mathcal{H}$ over an appropriate minimizing class. Even though the main reason for considering this problem is to solve Problem 2, the  result turns out to be  interesting by itself.

\subsection{The approximation problem for an orthogonally decomposed  $\calH$}\label{sec:approx-on-H}
\

\medskip
In this subsection, we give a solution for an abstract version of the approximation problem for general Hilbert spaces. The main result is a generalization of \cite[Theorem 6.1]{CM15} where the problem is solved for $\ell^2(\Z^d)$ and a decomposition of it arising from a partition of $\Z^d$. Although our result is for general Hilbert spaces and arbitrary orthogonal decompositions, the proof  follows the lines of that of \cite[Theorem 6.1]{CM15}. 
Let $\calH$ be a  separable Hilbert space and suppose that it is decomposed into an orthogonal sum as in \eqref{eq:descomposicion}, that is 
$$\mathcal{H}=\mathcal{H}_1\oplus\cdots\oplus \mathcal{H}_{\kappa},$$
where $\kappa\in \N$. 
Denote by $P_V$ the orthogonal projection of $\calH$ onto the closed subspace $V$, and for a subspace $S\subseteq \calH$, let us abbreviate $S_i:=P_{\calH_i}S$ for $1\leq i\leq \kappa$. 

The class of subspaces of $\calH$ we shall work with is given by 
$$\calD_\ell:=\{S\subseteq \calH: S \,\,\textrm{ is a subspace}, \,\, \dim (S)\leq\ell,\,\, S_i\subseteq S\,\,\, \forall\,\, 1\leq i\leq \kappa\}.$$
The reason for choosing this particular class of subspaces is not arbitrary. In fact, this is exactly the class we have to use to be able to relate the present situation to Problem 2. We shall explain this in more details when proving that Problem 2 has solution.   

The minimization problem we want to solve now is the following: given a set of data $\X=\{x_1, \ldots, x_m\}\subseteq \calH$ and $\ell\in \N$ we seek for a subspace $S^*\in\calD_\ell$  satisfying
\begin{equation}\label{eq:minimizing-on-H}
\sum_{j=1}^m \|x_j- P_{S^*} x_j\|^2 \le \sum_{j=1}^m \|x_j- P_{S} x_j\|^2,
\end{equation}
for any other $S\in\calD_\ell$.

Note that, since $\|x\|^2=\|P_Sx\|^2+\|x- P_{S} x\|^2$ for all $x\in \calH$ and any subspace $S\subseteq\calH$ , then $S^*$ satisfies \eqref{eq:minimizing-on-H} if and only if $S^*$ satisfies 
$$\sum_{j=1}^m \|P_{S^*} x_j\|^2 \geq\sum_{j=1}^m \|P_{S} x_j\|^2,\quad \textrm {for  every }\quad S\in\calD_\ell.$$
On the other hand, for every subspace $S$ of $\calH$ it holds that $S\subseteq \bigoplus_{i=1}^\kappa S_i$. Hence, the condition $S\in\calD_\ell$ is equivalent to $S=\bigoplus_{i=1}^\kappa S_i$.

Then,  we  observe that for  $S\in\calD_\ell,$
\[
\sum_{j=1}^m \|P_{S} x_j\|^2=\sum_{j=1}^m \|\sum_{i=1}^\kappa P_{S_i} x_j\|^2
=\sum_{j=1}^m \|\sum_{i=1}^\kappa P_{S_i} P_{\calH_i}x_j\|^2= \sum_{j=1}^m  \sum_{i=1}^\kappa\|P_{S_i} P_{\calH_i}x_j\|^2.
\]
At this point we see that the above calculation suggests the following strategy for finding $S^*$: since $\sum_{j=1}^m \|P_{S} x_j\|^2$ is as big as possible 
when  $\sum_{j=1}^m\|P_{S_i} P_{\calH_i}x_j\|^2$ is as big as possible, for every $1\leq i\leq \kappa$, we will take an optimal $S_i$ --given by Theorem \ref{thm:optimal-Hilbert}--  minimizing the data $\X_i:=\{P_{\calH_i}x_1, \ldots, P_{\calH_i}x_m\}$ for every $1\leq i\leq \kappa$ and then show that $S^*:=\bigoplus_{i=1}^\kappa S_i$ is the optimal subspace we are looking for. In this process we will have to deal with an  extra condition on the dimension of the $S_i's$ because $\ell\geq \dim(S^*)=\sum_{i=1}^\kappa\dim(S_i)$. 

Before stating and proving the main result of this section about the existence of an optimal $S^*$ solving \eqref{eq:minimizing-on-H} we need to set some notation.  
For $\X=\{x_1, \ldots, x_m\}\subseteq \calH$ and $1\leq i\leq\ell$, $G_i$ is the Gramian associated to the data set $\X_i=$$\{P_{\calH_i}x_1, \ldots, P_{\calH_i}x_m\}$, that is the matrix in $\C^m\times \C^m$ given by $(G_i)_{kj}=\langle P_{\calH_i}x_k, P_{\calH_i}x_j\rangle_\calH$. By $\la^i_1\geq\cdots\geq\la^i_m$ we denote its eigenvalues and by $y^i_1, \ldots, y^i_m\in \C^m$ the corresponding left-eigenvectors. Let $\Lambda$ be the set of all eigenvalues of $G_i$ for every $1\leq i\leq\ell$, i.e. $\Lambda=\{\la^i_j:\,1\leq i\leq\ell,\,1\leq j\leq m\}$
and let $\mu_1\geq\cdots \geq\mu_\ell$ be the $\ell$ biggest elements of $\Lambda$ ordered decreasingly. Then, we have that for every $s=1,\ldots,\ell$, $\mu_s= \la^{i_s}_{j_s}$ for some $1\leq i_s\leq\ell,\,1\leq j_s\leq m$. With this notation we define $h_1,\ldots,h_\ell\in\calH$ by
\begin{equation}\label{eq:def-generators-in-H}
h_s:=(\la^{i_s}_{j_s})^{-1/2}\sum_{k=1}^m y^{i_s}_{j_s}(k)P_{\calH_{i_s}}x_k,
\end{equation}
when $\mu_s=\la^{i_s}_{j_s}\neq0$ and $h_s=0$ otherwise. Here, $y^{i_s}_{j_s}(k)$ is the $k$-th entry of $y^{i_s}_{j_s}$. Note that $h_s$ is defined in terms of the $s$-th biggest eigenvalue of $\Lambda$, its associated eigenvector and the corresponding projection of the elements of $\X$. 

\begin{theorem}\label{thm:optimal-H}
Let $\ell\in\N$ and $\X=\{x_1, \ldots, x_m\}\subseteq \calH$. Keeping the above notation we have that there exists $S^*\in \calD_\ell$ satisfying \eqref{eq:minimizing-on-H}. Moreover, $S^*=\overline{\mbox{span} }\{h_1,\ldots, h_\ell\}$  and $\{h_1,\ldots, h_\ell\}$ forms a Parseval frame for $S^*$, where $h_i$ is defined by \eqref{eq:def-generators-in-H} for all $1\leq i\leq\ell$. 
\end{theorem}

\begin{proof}
We will follow the strategy described above. For this, we begin by defining 
$\calQ:=\{\alpha:=(\alpha_1,\ldots,\alpha_\kappa):\,\,\alpha_i\in\N\cup\{0\},\,\sum_{i=1}^\kappa\alpha_i\leq\ell\}$. Now, for a fixed $\alpha\in\calQ$, let $S^\alpha_i$ be the optimal subspace that minimizes the expression
$$\sum_{j=1}^m \|P_{\calH_i}x_j- P_{S} P_{\calH_i}x_j\|^2$$
over the class of subspaces of $\calH$ with dimension at most $\alpha_i$. The existence of such subspace $S^\alpha_i$ is guaranteed by Theorem \ref{thm:optimal-Hilbert}. Moreover, the same result asserts that $S^\alpha_i$ is generated by $h^\alpha_1,\ldots,h^\alpha_{\alpha_i}$ where $h^\alpha_r:=(\la^{i}_{r})^{-1/2}\sum_{k=1}^m y^{i}_{r}(k)P_{\calH_{i}}x_k$ if $\la^{i}_{r}\neq0$ and $h^\alpha_r=0$ otherwise, and also asserts that $\{h^\alpha_1,\ldots,h^\alpha_{\alpha_i}\}$ forms a Parseval frame for $S^\alpha_i$. 

Now for every $\alpha\in \calQ$, we consider $S^\alpha:=\bigoplus_{i=1}^\kappa S^\alpha_i$ where $S^\alpha_i$ is constructed as above. Let $\beta\in\calQ$ be such that $S^\beta$ is the subspace that minimizes $\sum_{j=1}^m \|x_j- P_{S^\alpha} x_j\|^2$ over $\alpha\in\calQ$ -- which exists because $\calQ$ is finite --, and set $S^*=S^\beta$. 
 
By construction, $S^*\in\calD_\ell$ and it is easy to see that $S^*$ satisfies \eqref{eq:minimizing-on-H}. Furthermore, since $\{h^\beta_1,\ldots,h^\beta_{\beta_i}\}$ forms a Parseval frame for $S^\beta_i$ and $\{S^\beta_i:\,1\leq i\leq \kappa\}$ are orthogonal subspaces, $\bigcup_{i=1}^\kappa\{h^\beta_1,\ldots,h^\beta_{\beta_i}\}$ forms a Parseval frame for $S^*$. What is left is to observe that $\bigcup_{i=1}^\kappa\{h^\beta_1,\ldots,h^\beta_{\beta_i}\}=\{h_1\ldots, h_\ell\}$, where each $h_s$ is as in \eqref{eq:def-generators-in-H}.
\end{proof}

\begin{remark}\label{rem:error-on-H}
 By Theorem \ref{thm:optimal-Hilbert}, if $\beta\in\calQ$ is as in the proof of Theorem \ref{thm:optimal-H}, for every $1\leq i\leq\kappa$, the error $\sum_{j=1}^m \|P_{\calH_i}x_j- P_{S^\beta_i} P_{\calH_i}x_j\|^2$ agrees with $\sum_{j=\beta_i+1}^m\la^i_j$. Then, the error $\sum_{j=1}^m \|x_j- P_{S^*}x_j\|^2$ is $\sum_{i=1}^\kappa\sum_{j=\beta_i+1}^m\la^i_j=\sum_{j=\ell+1}^m\mu_j$.
\end{remark}

\subsection{Optimal decomposable MI spaces}\label{sec:optimal-decomposable-MI}
\

\medskip

As we mentioned before, in the second version of our approximation problem we will minimize over the subclass of $\calC_\ell$, $\mathcal{C}_{\{\calH_1, \ldots, \calH_\kappa\}, \ell}$ (see \eqref{eq:class-with-extra}). Specifically, Problem 2  reads a follows:

{\bf Problem 2:} Let $\ell\in\N$ and $\{\calH_1, \ldots, \calH_\kappa\}$ a decomposition of $\calH$ as in \eqref{eq:descomposicion}. Given $\F=\{F_1, \dots, F_m\}\subseteq L^2(\Omega, \mathcal{H})$, find a $D$-MI space $M^*\in\mathcal{C}_{\{\calH_1, \ldots, \calH_\kappa\}, \ell}$ such that
\begin{equation}\label{Problema2}
\sum_{j=1}^m \|F_j- P_{M^*} F_j\|^2 \le \sum_{j=1}^m \|F_j- P_{M} F_j\|^2,
\end{equation}
for all  $M\in \mathcal{C}_{\{\calH_1, \ldots, \calH_\kappa\}, \ell}.$

Recall that in the situation of Problem 2 we have that   $\calH$ is  decomposed as in \eqref{eq:descomposicion}. Let   $\F=\{F_1,\dots,  F_m\}\subseteq L^2(\Omega, \mathcal{H})$ 
be the data set we want to approximate using the class $\mathcal{C}_{\{\calH_1, \ldots, \calH_\kappa\}, \ell}$ where $\ell\in\N$ is fixed. For a.e. $\w\in\W$, we will consider 
the optimal subspace $J^*(\w)\in \calD_\ell$ for the data $\F(\w)=\{F_1(\w),\dots,  F_m(\w)\}\subseteq \mathcal{H}$ given by Theorem \ref{thm:optimal-H}. These will give  point-wise solutions  that we will ``paste up'' to construct a solution for Problem 2. Roughly spiking,  we will show that $\w\mapsto J^*(\w)$ is a measurable range function 
and the associated MI space, $M^*$, given by Theorem \ref{thm:MI-range-function} is the solution we are looking for. Note that since $J^*(\w)\in \calD_\ell$ for a.e. $\w\in\W$, $P_{\calH_i}(J^*(\w))\subseteq J^*(\w)$ for every $1\leq\ i\leq\kappa$, thus by Lemma \ref{lem:extra-invariance-range-function}, $M^*$ will be decomposable with respect to  $\{\calH_1,\ldots, \calH_\kappa\}$.  This justified the choice of the class $\calD_\ell$ in Section \ref{sec:approx-on-H}.  

Using the notation of Section \ref{sec:approx-on-H} we have that for a.e $\w\in\W$, $J^*(\w)$ is generated by $\{\Phi_1(\w),$$ \ldots, \Phi_{\ell}(\w)\}$ where, for $s=1\ldots,\ell$,  
\begin{equation}\label{eq:generator-MI}
\Phi_s(\w):= \lambda_{j_s(\w)}^{i_s(\w)}(\w)^{-1/2} \sum_{k=1}^m \left(y_{j_s(\w)}^{i_s(\w)}(\w)\right)(k) \calP_{i_s} F_k(\w),
\end{equation}
if $\lambda_{j_s(\w)}^{i_s(\w)}(\w)\neq0$ and $\Phi_s(\w)=0$ otherwise.

Now, the solution for Problem 2 is provided by the upcoming theorem. 
\begin{theorem}\label{sol-problema2}
Let $\ell\in\N$ and $\F=\{F_1, \ldots, F_m\}\subseteq L^2(\Omega, \mathcal{H}).$ Then $M^*=M_{D}(\Phi_1, \dots, \Phi_{\ell})\in \mathcal{C}_{\{\calH_1, \ldots, \calH_\kappa\}, \ell}$ with $\Phi_s$ defined a.e. $\w\in\W$ by \eqref{eq:generator-MI},
is a solution of Problem 2. Moreover, $\{\Phi_1, \dots, \Phi_{\ell}\}$ is a uniform Parseval frame for $J^*$, where $J^*$ is the measurable range function associated to $M^*$.  
\end{theorem}

\begin{proof}
The measurability of the functions $\Phi_1, \dots, \Phi_{\ell}$ is a consequence of  the measurability of   the eigenvalues $\lambda_{j_s(\w)}^{i_s(\w)}(\w)$ and the eigenvectors 
$y_{j_s(\w)}^{i_s(\w)}(\w)$, and this follows  readily from the argument given in \cite[Section 4]{CM15}. Furthermore, the same calculations we used to obtain \eqref{eq:L2-generators} shows that $\Phi_1, \dots, \Phi_{\ell}\in L^2(\W, \calH)$.
Therefore, by Theorem \ref{thm:MI-range-function}, $J^*(\w)=\overline{\mbox{span} }\{\Phi_s(\w)\colon 1\leq s\leq \ell\}$ is the measurable range function associated to the MI space $M^*= M_{D}(\Phi_1, \dots, \Phi_{\ell})$. As we mentioned at the beginning of this section, $J^*(\w)\in \calD_\ell$  by construction and then, by Proposition \ref{lem:extra-invariance-range-function}, $M^*\in \mathcal{C}_{\{\calH_1, \ldots, \calH_\kappa\},\ell}$.

To show that $M^*$ solves Problem 2, we first observe that if $M\in\mathcal{C}_{\{\calH_1, \ldots, \calH_\kappa\}, \ell}$ with associated range function $J$, then by Proposition \ref{lem:extra-invariance-range-function}, $J(\w)\in\calD_\ell$ for a.e. $\w\in\W$. 
Thus, by construction and Theorem \ref{thm:optimal-H}, we have that for a.e. $\w\in\W$
$$\sum_{j=1}^m\|F_j(\w)-P_{J^*(\w)}(F_j(\w))\|^2_\calH\leq 
\sum_{j=1}^m\|F_j(\w)-P_{J(\w)}(F_j(\w))\|^2_\calH.$$
Then, integrating over $\W$ and using Lemma \ref{lem:helson} we obtain 
$$\sum_{j=1}^m\|F_j-P_{M^*}F_j\|^2\leq 
\sum_{j=1}^m\|F_j-P_{M}F_j\|^2,$$
which says that $M^*$ is a solution for Problem 2. 

Finally, since by Theorem \ref{thm:optimal-H}, $\{\Phi_1(\w), \ldots, \Phi_{\ell}(\w)\}$ is a Parseval frame for $J^*(\w)$  for a.e. $\w\in\W$, we conclude that $\{\Phi_1, \dots, \Phi_{\ell}\}$ is a uniform Parseval frame for $J^*$. 
\end{proof}

\begin{remark}\noindent
 \begin{enumerate}
  \item Taking into account Remark \ref{rem:error-on-H}, we have that the approximation error \mbox{$\sum_{j=1}^m\|F_j-P_{M^*}F_j\|^2$} is exactly $\sum_{i=1}^\kappa\sum_{j=\beta_i+1}^m \int_\W \la^{i(\w)}_{j(\w)}(\w)\,d\mu(\w)=\int_{\Omega} \sum_{j=\ell +1}^m \mu_j(\w)\, d\mu(\w).$
\item Note that the strategies for finding a solution for Problems 1 and 2 are similar. In both cases, the idea is to find point-wise solutions using previous result for general Hilbert spaces and then paste them up together to construct the optimal subspaces which solve Problem 1 and 2.  
\item Problem 2 can be solved using Problem 1 in different  way. Indeed, given  a data set
$\F=\{F_1, \dots, F_m\}\subseteq L^2(\Omega, \mathcal{H}),$ we can consider  the data set containing the elements  of  $\F$ but split according to the decomposition of $\calH$ given by 
 \eqref{eq:descomposicion},  
$\widetilde{\F}:=\{\calP_1 F_1,\dots, \calP_1 F_m,\dots,\calP_{\kappa} F_1,\dots, \calP_{\kappa} F_m\}$$\subseteq L^2(\Omega, 
 \mathcal{H}).$
It can be proven that the solution of Problem 1 constructed for  the data $\widetilde{\F}$ is a solution of Problem 2 for the data set $\F$. 
 \end{enumerate}
\end{remark}

\section{Application to SI spaces in LCA groups}\label{sec:shift-invariant}
In this section we shall see how the result about MI spaces can be used to obtain optimal shift-invariant subspaces for a given set of data. 
We shall work in $L^2(G)$, where $G$ is a second countable locally compact abelian group (LCA group for short) with operation written additively and the measure involved is the Haar measure on $G$ which we denote by $m_G$. Furthermore, we consider  a uniform lattice $H$ on $G$, that is a countable discrete subgroup $H$ of $G$ such that $G/H$ is compact, and  the translation operators along elements of $H$, $T_h$, defined by $T_hf(x):=f(x-h)$, for $f\in L^2(G)$ and a.e. $x\in G$. 
We say that a closed subspace $V\subseteq L^2(G)$ is $H$-invariant if for every $f\in V$, $T_hf\in V$ for all $h\in H$. 

Now, given any (at most countable) set of functions $\A\subseteq  L^2(G)$ , the space $V:=S_H(\A)$ where $S_H(\A)=\overline{\mbox{span} }\{T_h\phi:\,h\in H,\,\phi\in\A\}$ is an $H$-invariant space that we call the $H$-invariant space generated by $\A$. In this section we focus  on finitely generated $H$-invariant spaces, that is, when $\A$ is a finite set of $L^2(G)$. For a finitely generated $H$-invariant space $V$  we define   its length $\ell(V)$ as the minimum number of functions we need to generate it, more precisely,
$$\ell(V):=\min\{n\in\N: \, \exists\,\phi_1,\ldots,\phi_n\in L^2(G),\mbox{ such that } V=S_H(\phi_1,\ldots,\phi_n)\}.$$

\subsection{Optimal $H$-invariant spaces}\label{optimal-H}
\

\medskip
The first approximation result we will prove concerns the optimal finitely generated $H$-invariant space of length at most $\ell$ that best fits   a given set of data. To properly state the result we define, for a given $\ell\in\N$ the class  of $H$-invariant spaces $\calV_\ell$ as
$$\calV_\ell:=\{V\subseteq L^2(G):\, V \mbox{ is a finitely generated $H$-invariant space  with }\ell(V)\leq\ell\}.$$

\begin{theorem}\label{thm:optimal-SIS}
Let $\ell\in\N$ and $\F=\{f_1,\ldots,f_m\}\subseteq L^2(G)$ be a set of data. Then there exists $V^*\in\calV_\ell$ such that 
\begin{equation*}
 \sum_{j=1}^m\|f_j-P_{V^*}f_j\|^2_2\leq  \sum_{j=1}^m\|f_j-P_{V}f_j\|^2_2\quad\textrm{for all }\quad V\in\calV_\ell. 
\end{equation*}
Moreover, there exists a generator set for $V^*$,  $\{\phi_1,\ldots,\phi_\ell\}$ such that $\{T_h\phi_j:\,h\in H,\, 1\leq j\leq\ell\}$ is a Parseval frame for $V^*$. 
\end{theorem}

To prove the above theorem we shall establish a one-to-one correspondence with MI spaces and use Theorem \ref{thm-problema1}. For this, we first need to recall some notions and properties about LCA groups. 

By $\wh G$ we denote the dual group of $G$, that is the set of continuous characters of $G$ and by $m_{\wh G}$ its Haar measure. We use the notation $(x,\g)$ for the complex value that the character $\g$ takes at $x$. For every $x\in G$, $e_x:\wh G\to \C$ is the character on $\wh G$ induced by $x$, i.e. $e_x(\g):=(x,\g)$ for all $\g\in\wh G$. 
For  a subgroup $K$ of $G$, we write $K^*$ for its annihilator which is the closed subgroup of $\wh G$ given by $K^*=\{\g\in\wh G: (x,\g)=1\,\, \forall\,x\in K\}$. In our case, since we have chosen the subgroup $H$ to be such that $G/H$ is compact, by  \cite[Lemma 2.1.3]{Rud62}, its annihilator $H^*$ is discrete. Moreover, since $G/H$ is metrizable, $H^*$ is countable, \cite[Theorem 24.15]{HR79}. We then fix 
$\W\subseteq \wh G$ a measurable section of the quotient $\wh G/H^*$, whose existence  
is provided by \cite[Lemma 1.1]{Mac52}. Since additionally $H$ is discrete, $\W$ can be chosen to be relatively compact and thus of finite measure (see \cite[Lemma 2]{KK98}). 
The Fourier transform of $f\in L^1(G)$ is defined by $\wh f(\g)=\int_G f(x)(-x,\g)\,dm_G(x)$ and  it can be extended to an operator from $L^2(G)$ to $L^2(G)$ which, for a proper normalization of $m_G$ and $m_{\wh G}$, turns out to be an isometric isomorphism. 

The upcoming proposition was proven in  \cite[Proposition 3.3]{CP10} and provides an isometric isomorphism between $L^2(G)$ and the vector-valued space $L^2(\W, \ell^2(H^*))$.
\begin{proposition}\label{prop:fiber}
 The fiberization mapping $\T:L^2(G)\to L^2(\W, \ell^2(H^*))$ defined by
 $$\T f(\w)=\{\widehat{f}(\w+\delta)\}_{\delta\in H^*}$$
 is an isometric isomorphism and it satisfies
 $\T T_hf(\w)=(-h, w)\T f(\w)$ for all $f\in L^2(G)$, all $h\in H$ and a.e. $\w\in\W$.
\end{proposition}

The fiberization isometry of Proposition \ref{prop:fiber} not only gives that  $L^2(G)$ is isomorphic to $L^2(\W, \ell^2(H^*))$ but also provides the correspondence between $H$-invariant spaces in $L^2(G)$ and MI spaces of $L^2(\W, \ell^2(H^*))$. Here the underlying determining set is $D=\{e_h\chi_{\W}\}_{h\in H}$, where $\chi_\W$ is the characteristic function of $\W$, \cite[Corollary 3.6]{BR14}.  Then, since $\T T_hf=e_{-h}\T f$ for all $f\in L^2(G)$, we have that $V\subseteq L^2(G)$ is an $H$-invariant space if and only is $\T V\subseteq L^2(\W, \ell^2(H^*))$ is an MI space with respect to $D$. This fact allows us to prove Theorem \ref{thm:optimal-SIS}.

\begin{proof}[Proof of Theorem \ref{thm:optimal-SIS}]
Let $\T$ be the isomorphism of Proposition \ref{prop:fiber} and consider the set 
$\{\T f_1,\ldots,\T f_m\}\subseteq  L^2(\W, \ell^2(H^*))$. As we discussed before, $H$-invariant spaces are in one-to-one correspondence with  $D$-MI spaces under $\T$, where $D=\{e_h\chi_{\W}\}_{h\in H}$. Moreover, if $\A\subseteq L^2(G)$ is an at most countable set, then  $V=S_H(\A)$ if and only if $\T V=M_D(\T\A)$ where $\T\A:=\{\T\phi:\,\phi\in\A\}$. In particular this shows that $V\in\calV_\ell$ if and only if $\T V\in \calC_\ell$ where $\calC_\ell$ is as in \eqref{eq:class}. 

Let $M^*=M_D(\Phi_1,\ldots, \Phi_\ell)\in \calC_\ell$ be an optimal MI space for the data in $L^2(\W, \ell^2(H^*)),$ 
$\{\T f_1,\ldots,\T f_m\},$ which exists due to Theorem \ref{thm-problema1}. Now, set $V^*:=\T^{-1}M^*\in\calV_\ell$ and consider any $V\in \calV_\ell$. Then, by Theorem \ref{thm-problema1} and Proposition \ref{prop:fiber} we have
\begin{align*}
\sum_{j=1}^m\|f_j-P_{V^*}f_j\|^2_2&=\sum_{j=1}^m\|\T f_j-\T P_{V^*}f_j\|^2\\
&=\sum_{j=1}^m\|\T f_j- P_{M^*}\T f_j\|^2\\
&\leq\sum_{j=1}^m\|\T f_j- P_{\T V}\T f_j\|^2=\sum_{j=1}^m\|f_j-P_{V}f_j\|^2_2.\\
\end{align*}
Moreover, since $M^*=M_D(\Phi_1,\ldots, \Phi_\ell)$, $V^*=S_H(\phi_1,\ldots, \phi_\ell)$, where $\phi_j:=\T^{-1}\Phi_j$ for all $1\leq j\leq\ell$. Recall that by Theorem \ref{thm-problema1} $\{\Phi_1,\ldots, \Phi_\ell\}$ is a uniform Parseval frame for the range function associated to $M^*$,  then, using \cite[Theorem 4.1]{CP10} we have that $\{T_h\phi_j:\,h\in H,\, 1\leq j\leq\ell\}$ is a Parseval frame for $V^*$, and this completes the proof.
\end{proof}

\begin{remark}\noindent
 \begin{enumerate}
  \item It is well known that there are LCA groups which do not have any uniform lattice, (see \cite{ BR14, KK98}). However, there is an easy way to construct  groups that do have uniform lattices. For a given  discrete and countable LCA group $H$  and a compact  group $C$, simply take the LCA group $G=H\times C$ with the product measure. Then $H\times\{0\}$ is a uniform lattice on $G$.
  \item When instead of translations other operators -- such as dilations -- are involved,  an analogous result to Theorem \ref{thm:optimal-SIS} can be obtained. To be more precise, when a discrete abelian group $\Gamma$ is acting on a measure space $\mathcal{X}$, one can  consider closed  subspaces of $L^2(\mathcal{X})$ that are invariant under unitary operators arising from the group action. In this situation, it was proven in \cite{BHP15} that  the mapping that relates these invariant subspaces with MI spaces is the generalized Zak transform, introduced in \cite{HSWW10} (see \cite{BHP15} for details). Under this identification it can be proven that the version of  Theorem \ref{thm:optimal-SIS} adapted to this setting is also true.    
 \end{enumerate}
\end{remark}

\subsection{Optimal $H$-invariant spaces with extra invariance}
\

\medskip
Keeping the notation and hypotheses on $G$ and $H$ described above, we consider a closed  subgroup $\Gamma$ of $G$ containing $H$ and $H$-invariant spaces $V$ with extra invariance on $\Gamma$. This means that  $T_yV\subseteq V$ for all $y\in\Gamma$ and we also say that $V$ is $\Gamma$-invariant. 
These type of $H$-invariant spaces were completely characterized in \cite{ACP10} 
and they are in one-to-one correspondence with MI spaces in $L^2(\W, \ell^2(H^*))$ that are decomposable  with respect to the decomposition of $\ell^2(H^*)$ that we describe below.

Let $\calN$  be an at most  countable section for the quotient $H^*/\Gamma^*$. 
Then $H^*=\bigcup_{\sigma\in \calN}\Gamma^*+\sigma$ where the union is disjoint and therefore $\ell^2(H^*)=\bigoplus_{\sigma\in \calN}\ell^2(\Gamma^*+\sigma)$.

For $\sigma \in\mathcal{N}$
we define the set  $B_{\sigma}$ as
\begin{equation}\label{B-sigma}
B_{\sigma}=\W+\sigma+\Gamma^*=\bigcup_{\gamma^*\in \Gamma^*}(\W+\sigma)+\gamma^*.
\end{equation}
It is not difficult to see that $\{B_{\sigma}\}_{\sigma\in\calN}$ is a partition of  $\wh{G}.$


\begin{lemma}\label{lem:V-M-extra}
 Let $V\subseteq  L^2(G)$ be an $H$-invariant space and $\Gamma\subseteq G$  a closed subgroup containing $H$. If $\T$ is as in Proposition \ref{prop:fiber} and   $M:=\T V,$ then the following conditions are equivalent:
 \begin{enumerate}
  \item[(i)] $V$ is $\Gamma$-invariant.
  \item[(ii)] $M \subseteq  L^2(\W,\ell^2(H^*))$ is decomposable with respect to $\{\ell^2(\Gamma^*+\sigma)\}_{\sigma\in\calN}$.
 \end{enumerate}
\end{lemma}
\begin{proof}
For each $\sigma\in \calN$, let $P_\sigma: L^2(\W,\ell^2(H^*))\to L^2(\W,\ell^2(\Gamma^*+\sigma))$ be the orthogonal projection onto  $L^2(\W,\ell^2(\Gamma^*+\sigma))$. Now, for any $f\in L^2(G)$ define $f^{\sigma}$ via its Fourier transform as $\widehat{f^{\sigma}}=\chi_{B_{\sigma}}\widehat{f}$, where 
$\{B_{\sigma}\}_{\sigma\in\calN}$ is the partition of $\wh G$ given by \eqref{B-sigma}. Then we have that $\T f^{\sigma}(\w)=P_\sigma(\T f)(\w)$ for a.e. $\w\in\W$.  Therefore, item {\it(ii)} is equivalent to require that for every $f\in V$ and each $\sigma\in \calN$, 
$\T f^{\sigma}(\w)\in J(\w)$ for a.e. $\w\in\W$, where $J$ is the measurable range function associated to $M$ through Theorem \ref{thm:MI-range-function}. 
Thus, the result follows from \cite[Proposition 5.4]{ACP10}.
\end{proof}

We now restrict ourself to subgroups $\Gamma$ that are discrete. In that case, the section $\calN$ of $H^*/\Gamma^*$ is finite (see \cite[Remark 2.2]{CP12}). We consider the subclass of $\calV_\ell$ denoted by $\calV_{\Gamma, \ell}$
and defined by 
$$\calV_{\Gamma,\ell}:=\{V\in\calV_\ell:\, V \mbox{ has extra invariance on }\Gamma\}.$$
For this class we can obtain the analogous result to Theorem \ref{thm:optimal-SIS}. Its proof is a consequence of  Theorem \ref{sol-problema2} and Lemma \ref{lem:V-M-extra} and follows the lines of the proof of Theorem \ref{thm:optimal-SIS}. Thus, we omit it. 

\begin{theorem}\label{thm:optimal-SIS-with-extra-inv}
Let $H\subseteq G$ be a uniform lattice and $\Gamma$  a larger discrete subgroup of $G$ containing $H$. 
Let $\ell\in\N$ and $\F=\{f_1,\ldots,f_m\}\subseteq L^2(G)$ be a set of data. Then there exists $V^*\in\calV_{\Gamma,\ell}$ such that 
\begin{equation*}
 \sum_{j=1}^m\|f_j-P_{V^*}f_j\|^2_2\leq  \sum_{j=1}^m\|f_j-P_{V}f_j\|^2_2\quad\textrm{for all }\quad V\in\calV_{\Gamma,\ell}. 
\end{equation*}
Moreover, there exists a generator set for $V^*$,  $\{\phi_1,\ldots,\phi_\ell\}$ such that $\{T_h\phi_j:\,h\in H,\, 1\leq j\leq\ell\}$ is a Parseval frame for $V^*$. 
\end{theorem}

\begin{remark}
 Theorems \ref{thm:optimal-SIS} and \ref{thm:optimal-SIS-with-extra-inv} provide extensions to the setting of LCA groups of \cite[Theorem 2.1]{ACHM07} and \cite[Theorem 4.1]{CM15} respectively. 
\end{remark}



\section{Totally decomposable MI spaces and translation-invariant spaces }\label{sec:PW-MI-spaces}

In this section we shall work in the setting of Section \ref{sec:shift-invariant} where $G$ is a second countable  LCA group and we consider $H$-invariant spaces of $L^2(G)$ where $H$ is a uniform lattice on $G$. 
Our main result here relates {\it translation-invariant spaces}, that are  closed subspaces of $L^2(G)$ invariant under any translation on $G$, with {\it totally} decomposable MI spaces (see Definition \ref{def:decomposable-MI}). 
For this, we need to consider translation-invariant spaces. In particular we shall prove a result which describes the measurable subset of $\wh G$ associated to a translation-invariant space through the so-called Wiener Theorem.

We begin by proving a version of Wiener's Theorem in the setting of LCA groups. A version for $G=\R^d$  was proven in \cite{Rud66}, with a beautiful proof which uses elementary theory. 
Although  one can straightforwardly adapt it to the setting of LCA groups, we provide it here for the reader's convenience.   For a different proof in terms of range function we refer to \cite[Corollary 3.9]{BR14}.

\begin{proposition}\label{prop:wiener}
 Let $V\subseteq L^2(G)$ be a closed subspace. Then, the following are equivalent:
 \begin{enumerate}
  \item [(i)] $V$ is translation invariant. 
  \item [(ii)] There exists a measurable set $E\subseteq \wh G$ such that
  $V=\{f\in L^2(G)\colon \wh f=0 \textrm{ a.e. } E\}$.
 \end{enumerate}
\end{proposition}

\begin{proof}
 $(ii)\Rightarrow(i)$. This is a straightforward consequence of the fact that $ \wh{(T_xf)}=e_{-x}\wh f$ for every $f\in L^2(G)$ and every $x\in G$.
 
 $(i)\Rightarrow(ii)$. Let $P$ be the orthogonal projection of $L^2(\wh G)$ onto $\wh V=\{\wh f\,:\, f\in V\}$. Then, for $f,g\in L^2(\wh G)$ we have that $\langle f-Pf,Pg\rangle=0$. Since $V$ is translation invariant, for every $x\in G$ we have
$$
0=\langle f-Pf,Pg e_{x}\rangle=\int_{\wh G}(f-Pf)(\g)\overline{Pg(\g)} e_{-x}(\g)\,dm_{\wh G}
=((f-Pf)\overline{Pg})^\wedge(x),
$$
thus, $(f-Pf)\overline{Pg}=0$ which says that $f\overline{Pg}=Pf\overline{Pg}$ for all $f,g\in L^2(\wh G)$.
Swapping the roles of $f$ and $g$ in the above equalities, we also obtain that $g\overline{Pf}=Pg\overline{Pf}$ for all $f,g\in L^2(\wh G)$. Therefore, 
\begin{equation}\label{eq:f-P}
f\overline{Pg}=Pf\overline{g}\,\,\textrm{  for all }\,\,f,g\in L^2(\wh G). 
\end{equation}

Now, let $g_0\in L^2(\wh G)$  be such that $g_0>0$ a.e. $\wh G$ (see Proposition \ref{prop:not-null} in the Appendix) and call $Z$ the exceptional set of zero measure. Then, for $\phi:=\overline{Pg_0}/{g_0}$ defined a.e. on $\wh G\setminus Z$, by \eqref{eq:f-P} we have $Pf=\phi f$ for every $f\in L^2(\wh G)$.
Since $P$ is an orthogonal projection
$$\phi f=Pf=P^2f=\phi^2 f\textrm{  for all }\,\,f\in L^2(\wh G),$$
and choosing $f=g_0$ we can conclude that $\phi^2=\phi$ a.e. $\wh G\setminus Z$. This implies that $\phi$ takes that values $0$ and $1$ a.e. $\wh G\setminus Z$. 

Let $E\subseteq\wh G\setminus Z$ be the set on which $\phi=0$ a.e. Then, $f\in \wh V$ if and only if $f=Pf=\phi f$ and the latter holds if and only if $f=0$ a.e. $E$.
\end{proof}

Due to the separability of $L^2(G)$, for every translation-invariant space $V$, which in particular is an $H$-invariant space,  there exists an at most countable set $\A\subseteq L^2(G)$ such that $V=S_H(\A)$. Using this description   of $V$, we can describe the set $E$ of Proposition \ref{prop:wiener} associated to $V$ in terms of  its generators as an $H$-invariant space.

\begin{corollary}\label{cor:ti-ppal}
 Let  $V\subseteq L^2(G)$ be a translation-invariant space and $\mathcal{A}\subseteq L^2(G)$  an at most countable set such that $V=S_H(\mathcal{A})$. 
Then $V=\{f\in L^2(G)\,:\, \wh f=0 \textrm{ a.e. } E\}$ with 
 $ E=\bigcap_{\phi\in\mathcal{A}} \{\widehat{\phi}=0\}$ up to a measure zero set.
\end{corollary}
\begin{proof}
By Proposition \ref{prop:wiener}, there exists a measurable set $\wt E\subseteq \wh G$ such that  $V=\{f\in L^2(G)\,:\, \wh f=0 \textrm{ a.e. } \wt E\}$. 
Let us call 
$W=\{f\in L^2(G)\,:\, \wh f=0 \textrm{ a.e. } E\}$, where $E=\bigcap_{\phi\in\mathcal{A}} \{\widehat{\phi}=0\}.$ We will prove
that  $V=W$ (or equivalently that $\wt E = E$ up to a measure zero set).
Note that $V\subseteq W$ if and only if $E \subseteq \wt E$.

First,  since by definition of $E$, $E \subseteq \{\widehat\phi = 0\}$  for each $\phi \in \A $,  we have  $\A\subseteq W$. Thus, $\mbox{span}\{T_{h}\phi:  \phi \in \A,\;  h \in H\}\subseteq W$ because $W$ is translation invariant. Finally, taking closure we find that $V=S_H(\A)\subseteq \overline{W}=W$.

For the other inclusion, we see that since $\phi \in\A\subseteq  V$, then $\widehat\phi$  must be zero a.e. in $\wt E$.
Thus $\wt E \subseteq  \{\widehat{\phi}=0\}$ for every $\phi \in \A$ and then $\wt E \subseteq \bigcap_{\phi\in\mathcal{A}} \{\widehat{\phi}=0\} = E.$(i.e. $W\subseteq V.)$
\end{proof}

As the previous results show, translation-invariant  spaces are in correspondence with measurable sets of $\wh G$. One can also construct a translation-invariant space  by imposing an additional condition on its generators as an $H$-invariant space, as we shall see in the next proposition.
For this,  again as in Section \ref{sec:shift-invariant}, let $\W$ be a measurable section of the quotient $\wh G/H^*.$ The space $L^p(\W)$ is identified with $\{g\in L^p(\wh G)\colon g=0 \textrm{ a.e. } \W^c\}$, $1\leq p\leq +\infty$. In this situation, the set $\{\chi_\W e_h\}_{h\in H}$ defined by  the characters induced by  $H$,  forms  an orthogonal basis of $L^2(\W)$ (see \cite[Proposition 2.16.]{CP10}).

\begin{proposition}\label{prop:wiener-ppal}
Let $\mathcal{A} \subseteq L^2(G)$ be an at most countable set such that $\widehat{\phi}\in L^2(\Omega)$ $\forall \,\phi\in\mathcal{A},$ 
and consider the space $V=S_{H}(\mathcal{A})$. Then, $V$ is translation invariant. 
In particular, $V=\{f\in L^2(G):\,\widehat{f}=0\textrm{ a.e. }E\}$ with 
$ E=\bigcap_{\phi\in\mathcal{A}} \{\widehat{\phi}=0\}\subseteq \W$.
\end{proposition}

\begin{proof}
Let $x\in G$. Then, by \cite[Lemma 4.4]{CP10}, there  exists a sequence of trigonometric polynomials $\{p_n\}_{n\in \N}$ such that $p_n(\w)\to e_{-x}(\w)$ a.e. $\w\in \W$ with $\|p_n\|_{\infty}\leq C$ for every $n\in\N$.
 Now, for $g\in \widehat{V}$, since $|p_n-e_{-x}|^2|g|^2\leq (C+1)^2|g|^2$ on $\W$, by Dominated Convergence Theorem, we have that $p_ng\to e_{-x}g$ in $L^2(\W)$.
 Therefore, since $p_ng\in \widehat{V}$ for all $n\in\N$, we can conclude that $e_{-x}g\in \widehat{V}$. 
 
As a consequence, $V$ is translation invariant and by Corollary \ref{cor:ti-ppal} the result follows.
\end{proof}
 
We now state the main result of this section, where we characterize translation-invariant spaces in terms of MI spaces. We shall show that translation-invariant spaces are associated to MI spaces that are  totally  decomposable. 

\begin{theorem}
 Let $V\subseteq L^2(G)$ be an $H$-invariant space.  Let $\T:L^2(G)\to L^2(\W, \ell^2(H^*))$ be the fiberization mapping of Proposition \ref{prop:fiber}, 
and $\{\delta_k\}_{k\in H^*}$ be the canonical basis of $\ell^2(H^{*}).$ Then 
 the following conditions are equivalent:
 \begin{enumerate}
   \item [(i)] $M:=\T V$ is totally decomposable  with respect to $\{\mbox{span}\{\delta_k\}\}_{k\in H^*}$.
    \item [(ii)] $V$ is a translation-invariant space.
 \end{enumerate}
\end{theorem} 

\begin{proof}
 $(i)\Rightarrow(ii)$.
We know that $M_k\subseteq M$ for every $k\in H^*$ where $M_k=\calP_kM$ and $\calP_k(F)(\w)=P_{\textrm{span}\{\delta_k\}}(F(\w))$ for a.e. $\w\in\W$.
Since, $M=\bigoplus_{k\in H^*}M_k$, using that $\T$ is an isometric isomorphism and denoting $\T V_k=M_k$, we immediately have $V=\bigoplus_{k\in H^*}V_k$. Therefore,  
if we prove that each $V_k$ is translation invariant, we will be able to conclude that so is $V$.

Fix $k\in H^*$. 
Denote by $\W_k$ the translation by $k$ of $\W$, i.e. $\W_k=\W+k$. For $f\in V$, $f_k$ is the function defined by its Fourier transform as $\wh f_k:=\chi_{\W_k}\wh f$.
If $F\in M$ is $F=\T f$, then $F_k(\w)=\wh f(\w+k)\delta_k=\T f_k(\w)$ a.e. $\w\in\W$ and  we have that  $\wh f_k=\wh{(\T^{-1}F_k)}\in \wh{(\T^{-1}M_k)}=\wh V_k$. 
On the other hand, if $f\in V_k$, $F=\T f\in M_k$ and then $F_k=F$. Thus, $f_k=f$ and therefore $V_k=\{f_k:\,\, f\in V\}$.
Furthermore, since by Proposition \ref{prop:extra-invariance}, $M_k$ is a D-MI space with $D=\{e_h\chi_\W\}_{h\in H}$,  $V_k$ is $H$-invariant.
In particular, we obtain that for every $f\in V$, $W_k:=\overline{\mbox{span} }\{e_h\wh f_k:\,h\in H\}\subseteq \wh V_k$. 

Now, by Proposition \ref{prop:wiener-ppal} $(W_k)^{\vee}$ is translation invariant. Hence, for $x\in G$ and $f\in V$, $T_xf_k\in (W_k)^{\vee}\subseteq V_k$. Thus, $V_k$ is translation invariant and so is $V$.

 $(ii)\Rightarrow(i)$. We need to show that $M_k\subseteq M$ for all $k\in H^*$ which is equivalent to see that  $V_k\subseteq V$ for all $k\in H^*$. 
 
 Thus, let us fix $k\in H^*$. Since $V$ is translation-invariant, by Proposition \ref{prop:wiener} there exists $E\subseteq \wh G$ measurable such that $V=\{f\in L^2(G)\,:\, \wh f=0 \textrm{ a.e. } E\}$. Then, if $f\in V$, $\wh f_k=\chi_{\W_k}\wh f$ satisfies that
 $\wh f_k=0$ a.e. $E$ which implies that $f_k\in V$. Therefore, $V_k\subseteq V$ as we wanted to prove.
\end{proof}

\medskip
\section{Acknowledgments}
 
We are indebted to the anonymous referee for his/her meticulous report and fine comments.
We would like to thank Michael Cwikel and Mario Milman for the invitation to be part of this homage.

 \medskip

\appendix
\section{}\label{sec:appendix}
In this Appendix we provide a proof for the existence of a function in $L^2(G)$ which is never zero - up to, perhaps, a zero measure  set.  We believe that  this is known, but since we did not find a precise reference, we construct it here. 

\begin{proposition}\label{prop:not-null}
Let $G$ be a second countable LCA group. Then, there exists a function $f\in L^2(G)$ such that $f>0$ a.e. on $G$.  
\end{proposition}

\begin{proof}
From \cite[Lemma 4.19]{Fol95} one can conclude that for every compact set $K\subseteq G$ there exits $g\in L^2(G)$ such that $ g>0$ on $K$ and $g\geq 0$ on $G$. 

Since $G$ is second countable, in particular  it is $\sigma$-compact. Then, we have that $G=\bigcup_{ j\in\N}K_j$ where each $K_j$ is compact. We now consider the sets $\{F_j\}_{j\in\N}$ which are the disjoint versions of $\{K_j\}_{j\in\N}$, that is, $F_1:=K_1$ and $F_j:=K_j\setminus (\bigcup_{i=1}^{j-1}K_i)$ for $j\in\N$. Note that we still have $G=\bigcup_{ j\in\N}F_j$ where now the union is disjoint. 

For every $j\in\N$, let $g_j\in L^2(G)$ be such that $ g_j>0$ on $K_j$ and $g_j\geq 0$ on $G$ and consider $f_j:=(1/2^j)\chi_{F_j}({g_j}/\|g_j\|_2)$. Then, $f:=\sum_{j\in\N}f_j$ is the function we are looking for. Indeed. First note that
since $\{F_j\}_{j\in\N}$ is a partition for $G$, 
$$\|f\|_2^2=\sum_{j\in\N}\|f_j\|^2_2\leq \sum_{j\in\N}\frac1{4^j}<+\infty,$$
and therefore $f\in L^2(G)$. Finally, since $ g_j>0$ on $K_j$ for each $j\in\N$, $ f_j>0$ on $F_j$ for each $j\in\N$, and thus, $ f>0$ a.e. on $G$.  
\end{proof}

\bibliographystyle{amsplain}

\end{document}